\documentclass[12pt]{article}
\usepackage[a4paper, total={6in, 8in}]{geometry}
\usepackage[utf8]{inputenc}
\usepackage{hyperref}

\usepackage{pdflscape,youngtab,young}

\input xypic
\input xy
\xyoption{all}
\usepackage{tikz}
\usetikzlibrary{positioning, calc}
\usepackage{amsmath}
\usepackage{amssymb}
\usepackage{relsize}
\usepackage{amsthm}
\usepackage{enumerate}
\usepackage{mathrsfs}
\usepackage{mathtools}
\usepackage{xcolor}
\usepackage{xfrac}
\usepackage{pst-eucl}%
\usepackage{tkz-euclide}
\usepackage{tikz}
\usepackage{enumitem}
\usepackage{tcolorbox}
\usetikzlibrary{decorations.markings}
\usepackage{fancyhdr}
\tikzset{middlearrow/.style={
        decoration={markings,
            mark= at position 0.5 with {\arrow{#1}} ,
        },
        postaction={decorate}
    }
}

\usepackage[english]{babel}
\newcolumntype{C}{>{$}c<{$}} 

\newcommand{\Mod}[1]{{\hskip -6pt}\pmod{#1}}

\newtheorem{theorem}{Theorem}[section]
\newtheorem{corollary}{Corollary}[section]
\newtheorem{lemma}[theorem]{Lemma}
\newtheorem{definition}[theorem]{Definition}
\title{Generalized Sierpi\'nski Numbers}
\author{Michael Filaseta\\
Mathematics Department, 
University of South Carolina \\
Columbia, SC 29208, 
Email:  filaseta@math.sc.edu
\and
Robert Groth \\
Mathematics Department, 
University of South Carolina \\
Columbia, SC 29208, 
Email: rgroth@email.sc.edu
\and
Thomas Luckner\\
Mathematics Department, 
University of South Carolina \\
Columbia, SC 29208, 
Email: luckner@email.sc.edu}
\date{\today}

\begin{document}

\maketitle

\begin{abstract}
A Sierpi\'nski number is a positive odd integer $k$ such that $k \cdot 2^n + 1$ is composite for all
positive integers $n$. Fix an integer $A$ with $2 \le A$. We show that there exists a positive odd integer $k$ such that
$k\cdot a^n + 1$ is composite for all integers $a \in [2, A]$ and all $n \in \mathbb{Z}^+$.
\end{abstract}

\section{Introduction}
A \textit{covering system (on $n$)} is a finite collection of congruence classes $\mathcal{C}=\{n\equiv r_j\pmod{m_j}\}$ such that every integer is a member of at least one class in $\mathcal{C}$.  P.~Erd{\H o}s \cite{pe} used a covering system to show that there is an arithmetic progression of positive integers, none of which can be written as a prime plus a power of $2$. Another early application of covering systems is due to W.~Sierpi\'nski \cite{sierpinski} who showed that there is an arithmetic progression of positive integers $k$ having the property that $k \cdot 2^n + 1$ is composite for all positive integers $n$.  Specifically, Sierpi\'nski observed the implications
\[
\begin{array}{l l l l}
      n \equiv 1 \Mod{2}, & k \equiv 1 \Mod{3} & \Longrightarrow & k\cdot2^n+1\equiv0\Mod{3}\\
     n \equiv 2 \Mod{4}, & k \equiv 1 \Mod{5} & \Longrightarrow & k\cdot2^n+1\equiv0\Mod{5}\\
     n \equiv 4 \Mod{8}, & k \equiv 1 \Mod{17} & \Longrightarrow & k\cdot2^n+1\equiv0\Mod{17}\\
     n \equiv 8 \Mod{16}, & k \equiv 1 \Mod{257} & \Longrightarrow & k\cdot2^n+1\equiv0\Mod{257}\\
     n \equiv 16 \Mod{32}, & k \equiv 1 \Mod{65537} & \Longrightarrow & k\cdot2^n+1\equiv0\Mod{65537}\\
     n \equiv 32 \Mod{64}, & k \equiv 1 \Mod{641} & \Longrightarrow & k\cdot2^n+1\equiv0\Mod{641}\\
     n \equiv 0 \Mod{64}, & k \equiv -1
     \Mod{6700417} & \Longrightarrow & k\cdot2^n+1\equiv0\Mod{6700417}.
    \end{array}
\]
The congruences on $n$ on the left form a covering system.  The conditions on $k$ then force the implications to all hold.  From this, one can deduce then that if
\[
k\equiv 15511380746462593381 \Mod{3\cdot5\cdot...\cdot6700417},
\]
then $k \cdot 2^n + 1$ is divisible by one of the primes 
\[
3, \quad 5, \quad 17, \quad 257, \quad 65537, \quad 641, \quad  6700417
\]
for each positive integer $n$ and, hence, $k \cdot 2^n + 1$ is composite. 
An \textit{odd} positive integer $k$ having the property that $k \cdot 2^n + 1$ is composite for all positive integers $n$ is called a \textit{Sierpi\'nski number}. The condition that $k$ be odd was introduced at least in part to avoid the possibility of $k$ being a power of $2$ (see \cite{bckl} for more details). 
We note that the smallest known Sierpi\'nski number is $78557$, which was found by J.~Selfridge (unpublished).

The main purpose of this paper is to offer the following generalization of Sierpi\'nski's result.

\begin{theorem}\label{themainthm}
Fix $A\in\mathbb{Z}$ with $2\le A$. There exists an arithmetic progression of odd positive integers $k$ such that $k\cdot a^n + 1$ is composite for all $a\in[2,A]$ and all $n\in\mathbb{Z}^+$. 
\end{theorem}

\noindent
Further related results are described in the last section of the paper.  

Since for a given positive integer $k$, Dirichlet’s theorem implies there are infinitely many positive integers $a$ such that  $k a + 1$  is prime, we know that there cannot be a $k$ such that $k\cdot a^n + 1$ is composite for all $a\in \mathbb Z^+$ and all $n\in\mathbb{Z}^+$.  Thus, we cannot replace $a\in[2,A]$ in Theorem~\ref{themainthm} by $a\in \mathbb Z^+$.

For the proof of Theorem~\ref{themainthm}, we will show that there is an arithmetic progression of odd positive integers $k$ such that \textit{every sufficiently large} $k$ in the arithmetic progression satisfies $k\cdot a^n + 1$ is composite for all $a\in[2,A]$ and all $n\in\mathbb{Z}^+$. This will be sufficient as one merely needs consider a sub-arithmetic progression of such an arithmetic progression to obtain the theorem as stated.

Before proceeding, we comment that the work of A.~Brunner, C.~Caldwell, D.~Krywaruczenko, and C.~Lownsdale \cite{bckl} discusses the notion of $a$-Sierpi\'nski numbers which is different from the numbers $k$ considered in Theorem~\ref{themainthm}.  More relevant to this paper is their nice observation that if one removes the choice of $a = 2$ in Theorem~\ref{themainthm}, the result is easily established.  More precisely, 
if one takes $P$ to be the product of the primes $\le A$ and $k \equiv -1 \pmod{P}$, then for each integer $a \in [3,A]$ and $p$ a prime divisor of $a-1$, we have
$k \cdot a^n + 1 \equiv 0 \pmod{p}$.  Therefore, every sufficiently large 
$k \equiv -1 \pmod{P}$ has the property that $k \cdot a^n + 1$ is composite.  Capturing however the case $a = 2$ in 
Theorem~\ref{themainthm} as well seems considerably more difficult, which is what we are addressing in this paper.

We turn to terminology regarding congruence systems. Let 
\[
\mathcal{S}=\{x\equiv r_j \Mod{m_j} : 1 \le j \le s \}
\]
denote a finite congruence system on $x$. If $t$ is a member of any of the congruence classes in $\mathcal{S}$, we say $\mathcal{S}$ \textit{covers} $t$ (equivalently, $t$ is covered by $\mathcal{S}$). As is typical, if $t$ is a member of all of the congruence classes in $\mathcal{S}$, we say $\mathcal{S}$ is \textit{satisfied} by $t$ (equivalently, $t$ satisfies $\mathcal{S}$). 
Thus, we have
\begin{align*}
t \text{ is covered by }\mathcal{S} &\iff t \equiv r_j \Mod{m_j} \text{ for some } j \in \{ 1, 2, \ldots, s\},\\
\intertext{and}
t \text{ satisfies }\mathcal{S} &\iff t \equiv r_j \Mod{m_j} \text{ for all } j \in \{ 1, 2, \ldots, s\}.
\end{align*}
These definitions extend naturally to subsets of the integers.  Thus, if $T \subseteq \mathbb Z$, then $T$ is covered by a congruence system $\mathcal S$ if each element of $T$ is covered by $\mathcal S$ and $T$ satisfies $\mathcal S$ if each element of $T$ satisfies $\mathcal S$. Lastly, suppose $\mathcal{S}_1$ and $\mathcal{S}_2$ are two congruence systems. We will say $\mathcal{S}_1$ and $\mathcal{S}_2$ are \textit{compatible} if the set of integers that satisfy both $\mathcal{S}_1$ and $\mathcal{S}_2$ is nonempty. In other words, there exist integers $k$ that satisfy the system $\mathcal{S}_1\cup \mathcal{S}_2$. 

Cyclotomic polynomials also play an important role in our work. Recall the $m^{th}$ cyclotomic polynomial in $x$, denoted $\Phi_m(x)$, is the unique irreducible polynomial with integer coefficients that divides $x^m-1$ and does not divide $x^j-1$ for any $1\le j < m$. For any integer $p$, note the implication
\[p\mid\Phi_m(a)\implies
     p\:\mid(a^m-1)\implies a^m\equiv 1\Mod{p}.\]

\noindent This observation illuminates a connection between producing a Sierpi\'nski number (or a Sierpi\'nski-like number, where $a$ may take on a different value than 2) and the prime factors of $\Phi_m(a)$. Explicitly, given $n\equiv r\pmod{m}$ and $p\:\mid\Phi_m(a)$, we can be assured $k\cdot a^n+1$ will be divisible by $p$ by setting $k\equiv -a^{-r}\pmod{p}$. We will also utilize other facts about cyclotomic polynomials to help achieve our result.

\section{Preliminaries}
We give one definition and establish five lemmas before addressing the main result. The definition is primarily for ease of writing.

\begin{definition}
Let $\mathcal{K}=\{k\equiv r_j\pmod{m_j}\}$ be a congruence system on $k$. If $r_j\equiv 1\pmod{m_j}$ for each $j$, then we call $\mathcal{K}$ a $\textbf{1-system}$ (on $k$).
\end{definition}

\noindent
In other words, a 1-system is a congruence system where each congruence class may be represented with a common residue of 1. As an example, we have
\[\begin{array}{l}
k \equiv 1\Mod{2}\\
k\equiv 1\Mod{3}\\
k\equiv 1\Mod{4}.
\end{array}\]
By applying the Chinese Reminder Theorem, we see that the integers $k$ that satisfy a 1-system are the same as the integers $k$ that satisfy $k \equiv 1 \pmod{L}$ where $L$ is the least common multiple of the moduli in the 1-system.

\begin{lemma}\label{1-syslem} Fix $T\in\mathbb{Z}^+$. Let $a$ be an integer greater than or equal to 2.
Then there exists a $1$-system of $T$ congruences with prime moduli $p_1,...,p_T$ such that, for every $k$ satisfying this $1$-system, and every $n\in\mathbb{Z}^{+}$ where $n\not\equiv 0\pmod{2^T}$,
the expression $k\cdot a^n+1$ is divisible by one of $p_1,...,p_T$.
\end{lemma}

\begin{proof}
Let $p_j$ be a prime dividing $a^{2^{j-1}}+1$ where $1\le j \le T$. Here, the $p_j$ need not be distinct.  For each $j$, we have $a^{2^{j-1}}\equiv -1 \pmod{p_j}$ and $a^{2^{j}}\equiv 1\pmod{p_j}$.  For $n\equiv 2^{j-1}\pmod{2^j}$, we can write $n = 2^jt+2^{j-1}$ for some integer $t$.
Then we obtain
\[k\cdot a^{n}+1= k\cdot a^{2^jt+2^{j-1}}+1\equiv k(1)(-1)+1\equiv -k+1\Mod{p_j}. \]
Taking $k\equiv 1\pmod{p_j}$, we see that the expression $k\cdot a^n+1$ is divisible by $p_j$. Thus, the implications
    \[
    \begin{array}{l l l l} 
      n \equiv 1 \Mod{2}, & k \equiv 1 \Mod{p_1} & \Longrightarrow & k\cdot a^n+1\equiv0\Mod{p_1}\\
     n \equiv 2 \Mod{4}, & k \equiv 1 \Mod{p_2} & \Longrightarrow & k\cdot a^n+1\equiv0\Mod{p_2}\\
     n \equiv 4 \Mod{8}, & k \equiv 1 \Mod{p_3} & \Longrightarrow & k\cdot a^n+1\equiv0\Mod{p_3}\\
     \:\:\:\:\:\vdots & \:\:\:\:\:\vdots & &\:\:\:\:\:\:\:\:\:\:\vdots\\
     n \equiv 2^{T-1} \Mod{2^T}, & k \equiv 1 \Mod{p_{T}} & \Longrightarrow & k\cdot a ^n+1\equiv0\Mod{p_{T}}\\
    \end{array}
    \]
\noindent
all hold. Observe that the least common multiple of the moduli in the congruences on $n$ above is $2^T$. Further, every integer in $[1,2^T-1]$ is of the form $2^{j-1}t$ for some $j\in\{1,...,T\}$ and odd integer $t$. Note $2^{j-1}t\equiv 2^{j-1}\pmod{2^j}$, so each integer $1,...,2^{T}-1$ is satisfies one of the congruences on $n$ above.  Thus, these congruences on $n$ cover every integer except those that are $0$ modulo $2^T$. For $k$ satisfying the $1$-system above with moduli $p_1,\ldots,p_T$ and $n\in\mathbb{Z}^{+}$ satisfying $n\not\equiv0\pmod{2^T}$, we see that $k\cdot a^n+1$ will be divisible by some prime among $p_1,...,p_T$.
\end{proof}

\begin{lemma}\label{qpatchlem} Fix $T\in\mathbb{Z}$ with $T\ge2$. Let $q$ be an odd, positive integer such that $q\le T+1$. Let 
\[
L = \{\ell_1, \ldots, \ell_q \} \subseteq [0,T] \cap \mathbb{Z}.
\]
The congruence class $n\equiv 0\pmod{2^{T}}$ is covered by the congruence system 
\[
\mathcal{C}_0=\{n \equiv 2^T j \Mod{2^{\ell_j}q}\mid 1 \le j \le q \}.
\]
\end{lemma}

\begin{proof}
The numbers $2^T j$ for $1 \le j \le q$ run through a complete residue system modulo $q$.  The result follows on noting that $n \equiv 2^T j \pmod{2^{\ell_j}q}$ covers the integers that are $0$ modulo $2^T$ and $2^T j$ modulo $q$.
\end{proof}

The idea of using Lemma~\ref{1-syslem} and Lemma~\ref{qpatchlem}  originates from \cite{FFK}.

\begin{lemma}\label{lemfil}
Let $n$ and $m$ be positive integers such that $n>m$. If $n/m$ is not a power of a prime, then there exists polynomials $u(x),v(x)\in\mathbb{Z}[x]$ satisfying 
\[\Phi_n(x)u(x)+\Phi_m(x)v(x)=1.\]
If for some prime $p$ the quotient $n/m$ is a power of $p$, then there exists polynomials $u(x),v(x)\in\mathbb{Z}[x]$ satisfying 
\[\Phi_n(x)u(x)+\Phi_m(x)v(x)=p.\]
\end{lemma}

We omit the proof of Lemma~\ref{lemfil}. One can find the statement and proof of Lemma~\ref{lemfil} above in \cite{fillemma}; also, see \cite{dresden}. Relative to our work, the primary purpose of Lemma~\ref{lemfil} is to establish our next lemma. 

Our next lemma involves the odd prime factors of $\Phi_N(a)$ where $N\ge3$ and $a\ge2$ with $a$ even. We show such an odd prime factor exists for each choice of $N$ and $a$ before stating the lemma. Given $N\ge 3$ and $a\ge 2$ with $a$ even, the value of $\Phi_N(a)$ is odd. To guarantee $\Phi_N(a)$ has an odd prime factor, it suffices to show $|\Phi_N(a)|>1$. Recall 
\[
\Phi_N(a)=\prod_{\substack{1\le j \le N\\ \gcd(j,N)=1}}(a-\zeta_N^j)
\]
where $\zeta_N=e^{2\pi i/N}$. Note the $\zeta_N^j$ resides on the unit circle in the complex plane, and $\zeta_N^j \ne 1$ for $N\ge3$ and $\gcd(j,N)=1$. We deduce that $|a-\zeta_j| > 1$ for each $j$ in the product above since $a\ge2$. Hence, $|\Phi_N(a)| > 1$, and $\Phi_N(a)$ has an odd prime factor.

\begin{lemma}\label{distinctplem}
Fix an even, positive integer $a$. The largest, necessarily odd, prime divisor of $\Phi_{2^jq}(a)$ as $j$ varies among the nonnegative integers and $q$ varies over the odd primes are distinct.
\end{lemma}

\begin{proof}
Let $j,\,j'$ be nonnegative integers and $q,\,q'$ be odd primes. Let $N=2^jq$ and $N'=2^{j'}q'$. Observe that $N=N'$ if and only if $j=j'$ and $q=q'$. Recall $\Phi_{N}(a)$ and $\Phi_{N'}(a)$ are odd, and neither is equal to $\pm1$. Let $p$ and $p'$ be the largest prime divisors of $\Phi_{N}(a)$ and $\Phi_{N'}(a)$ respectively. It suffices to show $p\neq p'$ when $j\neq j'$ or $q\neq q'$.  

Suppose $q\neq q'$. Without loss of generality, let $N\ge N'$. Then $N/N'$ will not be a power of a prime, and, by Lemma~\ref{lemfil}, there exists $u(x),v(x)\in\mathbb{Z}[x]$ such that
\[\Phi_{N}(x)\,u(x)+\Phi_{N'}(x)\,v(x)=1.\]
By evaluating the above expression at $x=a$, we can see $\Phi_{N}(a)$ and $\Phi_{N'}(a)$ are relatively prime. Thus $p\neq p'$. 

Now suppose $q = q'$ and $j\neq j'$. Then $N/N'$ will be a power of 2, and, by Lemma~\ref{lemfil}, there exists $u(x),v(x)\in\mathbb{Z}[x]$ such that
\[\Phi_{N}(x)\,u(x)+\Phi_{N'}(x)\,v(x)=2.\]
By evaluating the above expression at $x=a$, we can see the greatest common divisor of $\Phi_{N}(a)$ and $\Phi_{N'}(a)$ is at most 2. Recall $p$ and $p'$ are necessarily odd. Since $\Phi_{N}(a)$ and $\Phi_{N'}(a)$ may have no common factors greater than 2, we deduce $p\neq p'$.
\end{proof}

Through a modified argument, one may show the conclusion of Lemma~\ref{distinctplem} is true for both even and odd values of $a$. However, for our work, we need only apply the result when $a$ is even.

\section{Proof of Theorem~\ref{themainthm}}

We find a finite set of primes $\mathcal M$ and an arithmetic progression of $k$'s satisfying the following property. For each such $k$, each $a\in[2,A]$ and each $n\in\mathbb{Z}^{+}$, there is a prime in $\mathcal M$ that divides $k\cdot a^n+1$.
To do this, we impose  a congruence system on $k$ that ensures $k\cdot a^n+1$ will be divisible by some prime in $\mathcal M$ for each $n\in\mathbb{Z}^{+}$. The simplest case is when $a$ is odd. If $a$ is odd, then we take $2$ to be in $\mathcal M$ and $k\equiv 1\pmod{2}$ ensuring $k\cdot a^n+1\equiv0\pmod{2}$ for every $n \in \mathbb Z^{+}$.  We also then have that $k$ is odd.

Now suppose $a$ is even.  Fix $T \in \mathbb Z^{+}$.
For each even $a$, we will impose a congruence system $\mathcal{K}_a$ on $k$ to ensure $k\cdot a^n+1$ is divisible by some odd prime in $\mathcal M$ for each $n\not\equiv 0\pmod{2^T}$, and then a second congruence system $\mathcal{L}_a$ on $k$ to ensure $k\cdot a^n+1$ is divisible by some odd prime in $\mathcal M$ for each of the remaining $n\equiv 0 \pmod{2^T}$. Thus, any $k$ satisfying both $\mathcal{K}_a$ and $\mathcal{L}_a$ will ensure $k\cdot a^n+1$ is divisible by some prime in $\mathcal{M}$ for each $n\in\mathbb{Z}^{+}$. To guarantee the desired $k$ exists, the systems $\mathcal{K}_a$ and $\mathcal{L}_a$ must be compatible as $a$ varies among the even integers in $[2,A]$. We begin with the $\mathcal{K}_a$ congruence systems.

We take $T$ to be large enough so that $A<\log\log{T}$. For each even $a$ in $[2,A]$, 
 by Lemma~\ref{1-syslem}, there exist a set $\mathcal{P}_a$ of $T$ primes and a 1-system $\mathcal{K}_a$ of $T$ congruences on $k$ such that $k\cdot a^n+1$ is divisible by a prime in $\mathcal{P}_a$ for each $k$ satisfying $\mathcal{K}_a$ and $n\not\equiv 0\pmod{2^T}$.  Since $a$ is even, we take the primes in $\mathcal{P}_a$ to be odd.  In line with our approach outlined at the outset, we set 
 \[
 \mathcal P = \bigcup_{\substack{a \in [2,A]\\a \text{ even}}} \mathcal P_a
 \]
 and $\mathcal{P}\subseteq\mathcal{\mathcal{M}}$. Since all of the $\mathcal{K}_a$ are 1-systems, we are assured they are compatible as $a$ varies. In other words, there exists a solution to the congruence system
\[\mathcal{K}=\bigcup_{\substack{a\in[2,A]\\ a\text{  even}}} \mathcal{K}_a\]
on $k$. For $k$ satisfying $\mathcal{K}$, the expression $k\cdot a^n+1$ will be divisible by some prime in $\mathcal{P}\subseteq\mathcal{M}$, dependent on $a$ and $n$, for all even $a$ in $[2,A]$ and $n\not\equiv 0 \pmod{2^T}$. We make an observation about $\mathcal{K}$ to be used later.
Since $\mathcal{K}$ is the union of $\lfloor A/2\rfloor$ systems of $T$ congruences each, we have
\[
|\mathcal P| < T \cdot A<T\log\log T.
\]

Now we construct the $\mathcal{L}_a$ systems to address the $n\equiv 0\pmod{2^T}$.
Set $Q=\log T$. For each even $a\in[2,A]$, let $q = q(a)$ be an odd prime (not necessarily different for different $a$) less than or equal to $Q$. 
For each such $q$, we consider $q$ distinct integers $\ell_1, \ldots, \ell_q$ in $[0,T]$, where the $\ell_j$ need not differ as $q$ varies.  
By Lemma~\ref{distinctplem}, for each $j\in\{1,...,q\}$, there exists an odd prime $p_j'$ that divides $\Phi_{2^{\ell_j}q}(a)$, where the $p_j'$ are different for different choices of the pair $(\ell_j,q)$. 
Importantly, 
\[
a^{2^{\ell_j}q}\equiv 1\Mod{p'_j}
\] 
for each $j$. It follows there exists a set $\mathcal{P}_a'$ of $q$ distinct primes $p_1',...,p_{q}'$ such that the implications
    \[\begin{array}{l l l l} 
    n \equiv 2^T\cdot 1 \Mod{2^{\ell_1}q}, & k \equiv c_1 \Mod{p'_1} & \Longrightarrow & k\cdot a^n+1\equiv0\Mod{p'_1}\\
    n \equiv 2^T\cdot 2 \Mod{2^{{\ell_2}}q}, & k \equiv c_2 \Mod{p'_2} & \Longrightarrow & k\cdot a^n+1\equiv0\Mod{p'_2}\\
    n \equiv 2^T\cdot 3 \Mod{2^{{\ell_3}}q}, & k \equiv c_3 \Mod{p'_3} & \Longrightarrow & k\cdot a^n+1\equiv0\Mod{p'_3}\\
     \:\:\:\:\:\vdots & \:\:\:\:\:\vdots & &\:\:\:\:\:\:\:\:\:\:\vdots\\
    n \equiv 2^Tq \Mod{2^{\ell_q}q}, & k \equiv c_{q} \Mod{p'_{q}} & \Longrightarrow & k\cdot a ^n+1\equiv0\Mod{p'_{q}}\\
    \end{array}\]
all hold, provided each $c_j\equiv -a^{-{2^T}j}\pmod{p'_j}$. In line with our approach, denote the system on $k$ above as $\mathcal{L}_a$. By Lemma~\ref{qpatchlem}, the above congruence system on $n$ covers the integers that satisfy $n\equiv 0\pmod{2^T}$. Therefore, for every $k$ satisfying $\mathcal{L}_a$, the expression $k\cdot a^n+1$ will be divisible by some prime in $\mathcal{P}_a'$ whenever $n\equiv 0\pmod{2^T}$. We can be assured such $k$ exists for a fixed $q$, as the moduli of $\mathcal{L}_a$ are all distinct primes. To complete the proof, we need only construct such $\mathcal{L}_a$ that are compatible as $a$ varies, and also compatible with $\mathcal{K}$.

In constructing the $\mathcal{L}_a$ systems, let's consider the even $a\in[2,A]$ one at a time in an increasing fashion up to some even $b\le A$. For even $b\in[4,A]$, let
\[
\mathcal{Q}_b=\bigcup_{\substack{a \in [2,b-2]\\a \text{ even}}} \mathcal P_a'.
\]
Note $\mathcal{Q}_b$ is the set of prime moduli used in the systems $\mathcal L_2,...,\mathcal L_{b-2}$. In line with this description, set $\mathcal{Q}_2=\emptyset$. Recall each of these $\mathcal{L}_a$ systems requires $q(a)\le Q$ primes, and there are certainly less than $A$ of these systems. Thus, for any appropriate value of $b$, we have
\[|\mathcal{Q}_b|<Q\cdot A<\log T \log\log T\] by our choices of $Q$ and $T$. It follows that
\begin{equation}\label{pqabd}
|\mathcal{P}\cup\mathcal{Q}_b|<T\log \log T + \log T \log \log T < 2 T \log \log T. 
\end{equation}

Observe that we have the implication 
\[
\mathcal{Q}_a\cap \mathcal{P}_a'=\emptyset \text{  for each even $a\in[2,b]$} \implies \text{$\mathcal{L}_2,\dots,\mathcal{L}_b$ are compatible.}
\]
Similarly, if the primes used as moduli in $\mathcal{L}_a$ are distinct from those used in $\mathcal{K}$, then the systems $\mathcal{K}$ and $\mathcal{L}_a$ will be compatible as well. In summation, if we are able to sequentially construct the $\mathcal{L}_a$ systems up to some $b\le A$ such that the prime moduli used in each of the $\mathcal{L}_a$ systems are distinct from both $\mathcal{P}$ and $\mathcal{Q}_a$, then the systems $\mathcal{K}, \mathcal{L}_2,...,\mathcal{L}_b$ will all be compatible, as desired. In other words, we have the implication 
\[(\mathcal{P}\cup\mathcal{Q}_a)\,\cap\, \mathcal{P}_a'=\emptyset \text{  for each even $a\in[2,b]$} \implies \text{$\mathcal{K}, \mathcal{L}_2,...,\mathcal{L}_b$ are compatible.}\]

With plans to derive a contradiction, assume we cannot construct the $\mathcal{L}_a$ sequentially in this manner. Specifically, assume there exists a particular even $B\in[2,A]$ such that $\mathcal{P}\cup\mathcal{Q}_{B}$ necessarily overlaps with every possible collection of $q\le Q$ primes that may be chosen for moduli in $\mathcal{L}_{B}$. Recall, by Lemma~\ref{distinctplem}, for each odd prime $q\le Q$, 
each of the expressions $\Phi_{2^jq}(B)$ has a unique largest odd prime divisor as both $q$ and $j$ vary with $q \le Q$ an odd prime and $j \in [0,T] \cap \mathbb Z$. For a fixed odd prime $q \le Q$, denote the set of $T+1$ largest prime divisors of $\Phi_{2^jq}(B)$, with $0 \le j \le T$, by $\mathcal{D}_q$. Note that the sets $\mathcal{D}_q$ are disjoint. Also, for any $q$, any $q$ of the primes in $\mathcal{D}_q$ will suffice as moduli in constructing the congruences in $\mathcal{L}_B$ on $k$. Thus, by our assumption, each $\mathcal{D}_q$ contains at most $q-1$ primes not in $\mathcal{P}\cup\mathcal{Q}_B$ (that is, $q-1$ primes not previously used as moduli in $\mathcal{K},\mathcal{L}_2,...,\mathcal{L}_{B-2}$). We deduce that the remaining $(T+1)-(q-1)>T-Q$ distinct primes in each $\mathcal{D}_q$ must all have been used previously as moduli in $\mathcal{K},\mathcal{L}_2,...,\mathcal{L}_{B-2}$, and are therefore elements of $\mathcal{P}\cup\mathcal{Q}_B$. Thus, each set $\mathcal{D}_q$ contains at least $T-Q$ primes in $\mathcal{P}\cup\mathcal{Q}_B$. To arrive at a contradiction, we make an observation regarding the quantity of such sets $\mathcal{D}_q$, and the resulting implication on the size of $\mathcal{P}\cup\mathcal{Q}_B$.

Recall $A<\log\log T$ and $Q=\log T$. For sufficiently large $T$, and thus $Q$, the Prime Number Theorem guarantees the existence of at least $Q/(2\log Q)$ odd primes less than or equal to $Q$. Hence, there exist at least $Q/(2\log Q)$ sets $\mathcal{D}_q$ with odd primes $q \le Q$. We deduce now that
\[
|\mathcal{P}\cup\mathcal{Q}_B|>(T-Q)\cdot\frac{Q}{2\log Q}=(T-\log T)\cdot\frac{\log T}{2\log\log T},
\]
contradicting \eqref{pqabd} for sufficiently large $T$. Having arrived at a contradiction, it must be the case that we can sequentially construct the systems $\mathcal{L}_2,...,\mathcal{L}_A$ such that, at each even $a\in[2,A]$, we have 
\[
(\mathcal{P}\cup\mathcal{Q}_a)\,\cap\, \mathcal{P}_a'=\emptyset,
\]
implying the systems $\mathcal{K},\mathcal{L}_2,...,\mathcal{L}_A$ are all compatible. 

To complete the proof, let
\[
\mathcal{L}=\bigcup_{\substack{a\in[2,A]\\ a\text{  even}}} \mathcal{L}_a
\] and consider the congruence system $\mathcal{S}=\mathcal{K}\cup\mathcal{L}\cup \{k\equiv 1\pmod{2}\}$. Recall the congruence class $k\equiv 1\pmod{2}$ handles the cases where $a$ is odd. Note $k\equiv1\pmod{2}$ is compatible with both $\mathcal{K}$, a 1-system, and $\mathcal{L}$, which both only use odd primes as moduli. In line with our approach at the outset, let $\mathcal{M}$ denote the set of moduli in the system $\mathcal{S}$. Explicitly,
\[\mathcal{M}=\mathcal{P}\cup\mathcal{Q}_A\cup\mathcal{P}_A'\cup\{2\},
\]
and $\mathcal{M}$ is a finite collection of primes. The Chinese Remainder Theorem implies the existence of an arithmetic progression satisfying $\mathcal{S}$. For any of the $k$'s in the arithmetic progression satisfying the system $\mathcal{S}$, we can see the expression $k\cdot a^n+1$, for any $a\in[2,A]$ and $n\in\mathbb{Z}^{+}$, will be divisible by some prime in $\mathcal{M}$. By choosing $k$'s larger than any of the primes in $\mathcal{M}$, it follows $k\cdot a^n+1$ is composite for all $a\in[2,A]$ and all $n\in\mathbb{Z}^{+}$. Thus, we see that Theorem~\ref{themainthm} holds.

\section{Extensions}

There are a number of related results to Theorem~\ref{themainthm} which one can establish along the same lines.  With no changes to the covering systems or the congruences on $k$, the following result holds.

\begin{corollary}\label{kplus2n}
Fix $A\in\mathbb{Z}$ with $2\le A$. There exists an arithmetic progression of odd positive integers $k$ such that $k + a^n$ is composite for all $a\in[2,A]$ and all $n\in\mathbb{Z}^+$. 
\end{corollary}

To see this, let $A$ and $k$ be as in Theorem~\ref{themainthm}. Our proof of Theorem~\ref{themainthm} provides a finite set of primes $\mathcal M$ such that $k \cdot a^n + 1$ is divisible by some $p \in \mathcal M$ for each choice of integers $a \in [2,A]$ and $n > 0$.  Now, fix such an $a$ and $n$.  Let $m$ be a multiple of $\prod_{p \in \mathcal M} (p-1)$ which is larger than $n$. Then we deduce that for some $p \in \mathcal M$ we have $k \cdot a^{m-n} + 1$ is divisible by $p$.  Then $p \nmid a$ and $k + a^n \equiv k \cdot a^{m} + a^n \equiv 0 \pmod{p}$, showing $k + a^n$ is divisible by a prime from $\mathcal M$.  By taking a sub-arithmetic progression as before, Corollary~\ref{kplus2n} follows.

By modifying the argument for Theorem~\ref{themainthm} one can produce the analogous result for Riesel numbers, where a Riesel number is an odd integer $k$ such that $k\cdot 2^n-1$ is composite for all $n\in\mathbb{Z}^{+}$.  We state this as follows.

\begin{corollary}\label{riesel}
Fix $A\in\mathbb{Z}$ with $2\le A$. There exists an arithmetic progression of odd positive integers $k$ such that $k\cdot a^n - 1$ and $k - a^n$ are composite for all $a\in[2,A]$ and all $n\in\mathbb{Z}^+$. 
\end{corollary}

\noindent
In fact, if $U$ and $V$ are positive integers with $U < V$ such that the arithmetic progression in Theorem~\ref{themainthm} is $U + mV$, then the arithmetic progression $(V-U) + mV$ has the property that every 
integer $k$ in the progression is such that $k\cdot a^n - 1$ and $k - a^n$ are divisible by a prime in $\mathcal M$, as defined above, for all $a\in[2,A]$ and all $n\in\mathbb{Z}^+$.  The corollary follows.

By our construction of $k$ in Theorem~\ref{themainthm} (in particular, the system $\mathcal{K}$), the analogous Brier-type result is immediate for all $a\in[3,A]$, where a Brier number is a number $k$ that is simultaneously both Sierpi\'nski and Riesel.

\begin{corollary}\label{brier}
Fix $A\in\mathbb{Z}$ with $3\le A$. There exists an arithmetic progression of odd positive integers $k$ such that $k\cdot 2^n+1$ and $k + 2^n$ are composite for all $n\in\mathbb{Z}^{+}$ and $k\cdot a^n \pm 1$ and $k \pm a^n$ are composite for all $a\in[3,A]$ and all $n\in\mathbb{Z}^+$.
\end{corollary} 

\noindent 
Capturing the cases $k\cdot 2^n - 1$ and $k - 2^n$ as well appears more difficult.

The Chinese Remainder Theorem implies that the system of congruences $\mathcal{S}$ obtained for $k$ in our proof of Theorem~\ref{themainthm} is equivalent to a single congruence class of the form $L\pmod{M}$ where $M$ is the product of the primes in $\mathcal{M}$ and necessarily $\gcd(L,M)=1$. Note the arithmetic progression of $k$'s described in the statement of Theorem~\ref{themainthm} is a sub-arithmetic progression of the integers that satisfy $L\pmod{M}$, and thus may be expressed as $L'\pmod{M'}$ for appropriate integers $L'$ and $M'$ with $\gcd(L',M')=1$. As observed in  \cite{fjl}, we obtain the following consequences of the work of J.~Maynard \cite{maynard} (cf.~D.~Shiu \cite{shiu}).

\begin{corollary}\label{lastcor}
Fix $A\in\mathbb{Z}$ with $2\le A$. Let $K = K(A)$ denote the set of odd positive integers $k$ such that $k\cdot a^n + 1$ and $k + a^n$ are composite for all $a\in[2,A]$ and all $n\in\mathbb{Z}^+$. Then we have the following.
\begin{enumerate}
    \item 
    Let $p_j$ denote the $j^{th}$ prime number.  
    For $t\in\mathbb{Z}^{+}$ fixed, there exist positive integers $j$ such that $p_j, p_{j+1}, \ldots, p_{j+t-1}$ are all in $K$.  Furthermore, the set $J$ of such integers $j$ has positive density (depending on $A$ and $t$) in the set of positive integers.  In other words,
    \[
\liminf_{x \rightarrow \infty} \dfrac{|\{ j \in J: j \le x\}|}{x} > 0.
    \]
    \item 
    There exists a $C = C(A) \in \mathbb Z^+$ such that there are infinitely prime pairs $k$ and $k+C$ both in $K$.
\end{enumerate}
\end{corollary}

We recall that an idea from \cite{FFK} was used in our arguments.  It is perhaps of some interest to note that our arguments can be modified slightly to reflect the main result in \cite{FFK}.  For example, in Theorem~\ref{themainthm}, we could instead conclude that, with $R$ an arbitrary fixed positive integer, there is an arithmetic progression of odd positive integers $k$ such that $k^r \cdot a^n + 1$ is composite for all integers $a\in[2,A]$, all $n\in\mathbb{Z}^+$, and all integers $r \in [1,R]$.  Similarly, in Corollary~\ref{lastcor}, we can replace $K = K(A)$ with $K = K(A,R)$ where $K(A,R)$ denotes the set of odd positive integers $k$ such that $k^r \cdot a^n + 1$ and $k^r + a^n$ are composite for all $a\in[2,A]$, all $n\in\mathbb{Z}^+$, and all integers $r \in [1,R]$.


\begin{thebibliography}{9}

\bibitem{bckl}
A.~Brunner, C.~Caldwell, D.~Krywaruczenko, and C.~Lownsdale, \textit{Generalizing Sierpiński numbers to base $b$}, New Aspects of Analytic Number Theory, Proceedings of RIMS, Surikaisekikenkyusho Kokyuroku (2009), 69--79.

\bibitem{dresden}
G.~Dresden, 
\textit{Resultants of cyclotomic polynomials},
Rocky Mountain J.~Math.~42 (2012), 1461--1469.

\bibitem{pe}
P.~Erd\H os, 
\textit{On integers of the form $2^{k}+p$ and some related problems}, Summa Brasil.~Math.~2 (1950), 113--123.

\bibitem{fillemma}
M.~Filaseta,
\textit{Coverings of the integers associated with an irreducibility theorem of {A}. {S}chinzel},
In Number theory for the millennium, II (Urbana, IL, 2000),
pages 1--24, A K Peters, Natick, MA, 2002.

\bibitem{FFK}
M.~Filaseta, C.~Finch and M.~Kozek,
\textit{On powers associated with Sierpi\'nski numbers, Riesel numbers and Polignac's conjecture}, 
J.~Number Theory 128 (2008), 1916--1940.

\bibitem{fjl}
M.~Filaseta, J.~Juillerat and T.~Luckner,
\textit{Consecutive primes which are widely digitally delicate and Brier numbers}, \url{https://arxiv.org/abs/2209.10646}.

\bibitem{maynard}
J.~Maynard,
\textit{Dense clusters of primes in subsets},
Compositio Math. 152 (2016), 1517--1554.

\bibitem{shiu}
D.~K.~L.~Shiu, \textit{Strings of congruent primes}, J.~Lond.~Math.~Soc.~61 (2000), 359--373.

\bibitem{sierpinski}
W.~Sierpi\'nski,
\textit{Sur un probl\`eme concernant les nombres $k\cdot 2^{n}+1$}, 
Elem.~Math.~15 (1960), 73--74.


\end{thebibliography}
\end{document}